\documentclass[psamsfonts,12pt]{amsart}
\usepackage{amssymb,amsfonts}
\usepackage{subcaption}
\usepackage[all,arc]{xy}

\usepackage{enumerate}

\usepackage{mathrsfs}

\usepackage{pgf,tikz}
\usetikzlibrary{arrows}

\usepackage{amsrefs}

\usepackage{microtype}

\usepackage{tikz-cd}	
\usepackage{tikz}
\usepackage{hyperref}

\newtheorem{thm}{Theorem}[section]

\newtheorem{cor}[thm]{Corollary}

\newtheorem{prop}[thm]{Proposition}

\newtheorem{lemma}[thm]{Lemma}

\theoremstyle{definition}

\newtheorem{deff}[thm]{Definition}

\newtheorem{es}[thm]{Example}

\newtheorem{n}[thm]{Notation}

\newtheorem{introteo}{Theorem}
\newtheorem{quest}[introteo]{Question}

\theoremstyle{remark}

\newtheorem{remark}[thm]{Remark}

\DeclareMathOperator{\trop}{trop}

\newcommand{\val}{\mathfrak{v}}

\makeatletter

\let\c@equation\c@thm

\makeatother
\numberwithin{equation}{section}

\newcommand{\PP}{\mathbb{P}}

\newcommand{\ZZ}{\mathbb{Z}}
\newcommand{\tf}{\textbf}

\DeclareMathOperator{\codim}{codim}

  \DeclareMathOperator{\Grp}{\mathbb G}

\DeclareMathOperator{\conv}{conv}

\DeclareMathOperator{\F}{F}

\DeclareMathOperator{\pos}{pos}

\usepackage{pifont}

\title{Tropical Fano schemes}

\author{Sara Lamboglia}

\address{
Institut f\"ur Mathematik\\
 Goethe-Universit\"at Frankfurt\\
 Robert-Mayer-Str. 6-8\\
  60325 Frankfurt a. M.\\
   Germany\\
   }
   \email{lamboglia@math.uni-frankfurt.de}

\date{April 4, 2019}

\begin{document}

\begin{abstract}
We define a tropical version $\F_d(\trop X)$ of the Fano Scheme $\F_d(X)$ of a projective variety $X\subseteq \mathbb P^n$ and  prove that $\F_d(\trop X)$ is the support of a polyhedral complex contained in $\trop \Grp(d,n)$.
In general  $\trop \F_d(X)\subseteq \F_d(\trop X)$ but we construct linear spaces $L$ such that
$\trop \F_1(X)\subsetneq \F_1(\trop X)$
 and  show that   for a toric variety  $\trop \F_d(X)=\F_d(\trop X)$. 
\end{abstract}

\maketitle

%\tableofcontents

\markleft{}

\section{Introduction}

The classical Fano scheme of a projective variety $X\subseteq \mathbb P^n$ is the fine moduli space parametrising linear spaces  contained in $X$. It is denoted by  $\F_d(X)$, with $d$  the dimension of the linear spaces, and  is a subscheme of the Grassmannian  
$\Grp(d,n)$ of $d-$dimensional subspaces of $\mathbb P^n$. Fano schemes have been intensively studied because of 
their geometric properties.  Gino Fano \cite{Fano} first introduced these schemes  and  mostly considered the case of hypersurfaces. Then in the $70s$
 these schemes  have been  used to prove results on the irrationality of cubic threefolds \cite{CG,MURRE}. Recently there has been new interests for Fano schemes not only  in algebraic geometry 
  \cite{I-C,I-Z,I-S,I} but also in machine learning \cite{L-K} and geometric complexity theory \cite{Mulm}.

In this paper we study a tropical version of the Fano scheme. We investigate the structure of this tropical object  and  relations with the classical $F_d(X)$.

The first way of obtaining   a tropical version of $\F_d(X)$ is to consider its  tropicalization inside $\trop \Grp(d,n)$.  The points of  $\trop \F_d(X)$ are in correspondence with   the tropicalization of  the classical linear spaces  contained in $X$. However it is not true in general that a tropicalized linear space that lies in  $\trop X$ is the tropicalization of a classical linear space in $X$.  A famous example of this is  in \cite{Vig} where Vigeland proves that there are  smooth surfaces  in $\mathbb P^3$  of degree $3$ whose tropicalization contains infinitely many lines. Since there are only $27$  lines in the classical surfaces we deduce that these infinite tropical lines do not come from their tropicalization.

This leads us to  define  the  second  tropical version  of $\F_d(X)$ to be   the set  of tropicalized linear spaces  of dimension $d$ contained in $\trop X$. We call this the \textit{ tropical Fano scheme} and we denote it by $\F_d(\trop X)$.   We take the first steps in studying the structure and the properties of this object that can also be used   to investigate the classical Fano scheme.   

\begin{introteo}\label{theorem1}
Let $X$ be a  projective variety in ${\mathbb P}^n$. Then  the  tropical  Fano scheme $\F_d(\trop X)$ is  a polyhedral complex whose support is contained in    $\trop \Grp(d,n)$. Moreover if $X $ is a fan then $\F_d(\trop X)$ is a fan.
  \end{introteo}

The two  tropical versions of the Fano scheme come from two different constructions. The first is strictly linked to the algebraic variety and to its classical Fano scheme while the other only depends on the tropical variety $\trop X$. However we immediately observe that
\begin{equation}\label{cont}
 \trop \F_d(X)\subseteq\F_d(\trop X)
\end{equation}
and since  Theorem \ref{theorem1}   allows us to define a  dimension  for  $\F_d(\trop X)$ we obtain a  bound for the dimension of $F_d(X)$.
A natural question arises:
\begin{quest}\label{question}
For which varieties $X$ do we have $\trop \F_d(X)=\F_d(\trop X)$?
\end{quest}

We start by looking at the simplest algebraic varieties:  linear subspaces of $\mathbb P^n$. We then analyse the case of toric varieties embedded in  $\mathbb P^n$ via monomial maps. These are two examples where the tropicalization can be easily described. For a linear space $L$ the tropicalization is computed from the matroid associated to $L$. On the other hand a monomial map can be tropicalized to a linear map from $\mathbb R^r$ to $\mathbb R^n$ and its image is  the tropicalization of the toric variety associated to the monomial map (\cite[Corollary 3.2.13]{M-S}).

\begin{introteo}\label{introte}\textcolor{white}{ciap}
\begin{enumerate}
\item Let $n\ge 5$. If $L$ is a  generic $2-$dimensional plane  in $\mathbb P^{n}$  then  $$\trop \F_1(L)\subsetneq \F_1(\trop L).$$
\item If $X$ is a toric  variety in $\mathbb P^n$ then  $\F_d(\trop X)=\trop \F_d(X)$.
\end{enumerate}

\end{introteo}

The paper is structured as  follows. In Section \ref{sec1}   we define the tropical Fano scheme and we give a rigorous statement of Theorem \ref{theorem1} (Theorem \ref{Fanstructure}  and Corollary \ref{fanProperty}). We study the case of linear spaces in Section \ref{counter}. We prove the first part of  Theorem \ref{introte} in Theorem \ref{counterexample} and then  use it to prove the strict containment in (\ref{cont}) for a generic hypersurface.
In Section  \ref{toricv}  we analyse the case of toric varieties and we prove  the second part of Theorem \ref{introte} (Theorem \ref{equality}).
Finally in Section \ref{proofThm} we study the structure of $\F_d(\trop X)$. %\textcolor{red}{We give a necessary and sufficient condition for a tropical line to be contained in a tropical variety  (Lemma \ref{theone}) and then we prove  Theorem  \ref{Fanstructure}. For some of the results in this section we will need  some standard technical lemmas which we include in the Appendix.}

\subsection*{Acknowledgements}
The author  would like to thank Diane Maclagan  for useful suggestions and  a close reading, 
Nathan Ilten, Paolo Tripoli, 
Mar\'ia Ang\'elica Cueto  and Annette Werner  for helpful discussions and Melody Chan for her valuable comments that lead to the  improvement of  the final version of this paper.
The author was supported by EPSRC grant 1499803 and partially  by LOEWE research unit USAG.

\section{Definitions  of $\F_d(\trop X)$}\label{sec1}

In this section we set  notation and  define the tropical Fano Scheme $\F_d(\trop X)$ of the tropicalization of a projective variety $ X\subseteq \mathbb P^n$.\\

Let  $\Bbbk $ be a field with a surjective valuation $\val:\Bbbk^* \to \mathbb R$ (cf. Remark \ref{nonsurval}) and let $T^m$ be the torus $(\Bbbk^*)^{m+1}/\Bbbk^*$ contained in $\mathbb P^m$.
The tropical projective space $\trop \mathbb P^m$  is $(\overline{\mathbb R}^{m+1}\setminus \{(\infty,\ldots,\infty )\})/\mathbb R \textbf 1$ where $\overline{\mathbb R}$ denotes $ \mathbb R\cup\{\infty\}$ and $\mathbb R\textbf 1$ is the linear space spanned by the vector $(1,\ldots ,1)$. Let  $ O$ be a  $T^m$-orbit of $\mathbb P^m$. This is   the locus of points in $\mathbb P^m$ where $x_i=0$ for  every $i$ in the subset $I$ of all coordinates  and  $x_i\neq 0$ for $i\notin I$.
Its tropicalization  $\mathcal O :=\trop O$ is the locus of points $(x_0,\ldots,x_n)$ in
  $\trop \mathbb P^m$ where $x_i=\infty$ if and only if   $i\in I$. We refer to $ \mathcal O$ as an orbit of $\trop \mathbb P^m$. 

For any  projective variety $Y\subseteq \mathbb P^m$  the  tropicalization $\trop Y$  is given by the union of $\trop Y\cap \mathcal O:=\trop (Y\cap O)$ where $O$ is the unique orbit of $\mathbb P^m$ such that $\trop O=\mathcal O$ (see Section 6 in \cite{M-S}). If $Y$ is irreducible  and  $O$  is such that $\dim \overline{Y\cap O}=\dim Y$ then   $ Y\subseteq \overline O$ and $\trop Y=\overline{\trop Y\cap \mathcal O}$ in $\trop \mathbb P^m$ \cite[Theorem 6.2.18]{M-S}.\\

Let $\Grp(d,n)$ be the Grassmannian parametrising $d$-dimensional projective  subspaces in $\mathbb P^n$. We consider it  embedded via the Pl\"ucker map into $\PP^{\binom{n+1}{d+1}-1}$. Its tropicalization  $\trop \Grp(d,n)\subseteq \trop \PP^{\binom{n+1}{d+1}-1} $ parametrises tropicalized linear spaces of dimension $d$ in $\trop \mathbb P^n$ (\cite[Theorem 3.8]{Speyer}, \cite[Theorem 4.3.17 and Remark 4.4.2]{M-S},\cite{CC}). Hence it is possible to associate to each point $p$ of $\trop \Grp(d,n)$  a unique tropicalized linear space which we denote by $\Gamma_p$.
\begin{n}
Given two tropical varieties $\trop X,\trop Y$ we write 
$\trop X\subseteq \trop Y$ for the containment of the support of $\trop X$ in the support of $\trop Y$.
\end{n}

\begin{deff}
The \textit{tropical Fano scheme} is the set  $\F_d(\trop X)\subseteq \trop \Grp(d,n) $ defined by $$ \F_d(\trop X ):=\{p\in  \trop \Grp(d,n)  : \Gamma_p\subseteq \trop X \}. $$
\end{deff}

 In Section \ref{proofThm}  we prove   the following results:

\begin{thm}\label{Fanstructure}
	Let $X$ be a  projective variety in ${\mathbb P}^n$ and  $\mathcal O$ be an orbit  of $\trop \mathbb P^{\binom{n+1}{d+1}-1}$. Then   $\F_d(\trop X)\cap \mathcal O$ is  a polyhedral complex  whose support is contained in  the intersection   $\trop \Grp(d,n)\cap \mathcal O$.
\end{thm}
\begin{cor}\label{fanProperty}
	Consider a non empty intersection $\F_d(\trop X)\cap \mathcal O$
	and let $\mathcal O'$ be the unique  orbit of $\trop \mathbb P^n$  such that $\overline{\Gamma_p\cap \mathcal O'}=\Gamma_p$ for all $p\in \F_d(\trop X)\cap \mathcal O$. Then  $\F_d(\trop X)\cap \mathcal O$ is a fan if  $\trop X\cap \mathcal O' $ is  a fan.
\end{cor} 

\begin{remark}
Note that  $\trop \F_d( X)$ does not have the same property described in Proposition \ref{fanProperty}. There are varieties $X\subseteq \mathbb P^n$ such that $\trop (X\cap T^n)$ is a fan but $\trop \F_d(X)\cap \trop T^{\binom{n+1}{d+1}-1}$ is not. In the next section we give an explicit  example of this (Example \ref{nonfan}).
\end{remark}

\section{Linear spaces and generic hypersurfaces}\label{counter}

In this section we show that  there exist linear spaces and hypersurfaces for which the  containment $\trop \F_1(X)\subseteq \F_1(\trop X)$ is strict.
In Theorem \ref{counterexample} we prove that  if  $n\ge 5$ and  $L$ is a generic plane in $\mathbb P^n$ then   there exists a tropical line in $\trop L$ that is not realizable in $L$. We then  compute an  explicit example of a plane $L\subseteq \mathbb P^5$ with this property and we show that  $\dim \trop \F_1(L)<\dim  \F_1(\trop L)$.
 Finally in Proposition \ref{LinearObstruction} 
we prove that the containment is strict for a \textit{general} hypersurface $X$ whose tropicalization has the same support as a tropical hyperplane.

\begin{thm}\label{counterexample}
Let $n\ge 5$. There  exists a semi-algebraic  set in $\Grp(2,n)$ whose points are planes $L\subseteq \mathbb P^{n}$ such that $\trop \F_1(L)\subsetneq \F_1(\trop L)$.
\end{thm}

A semialgebraic subset of an algebraic variety $X$ is a subset of $X$ that can locally be defined by finitely many Boolean operators and inequalities of the form $\val (f) \le  \val (g)$ where $f,g$ are algebraic functions on $X$(\cite{Nic}). 
For example every set in $X$  that is Zariski open is also a semialgebraic set.

\begin{proof}[Proof of Theorem \ref{counterexample}]
Let $\mathcal L$ be the standard tropical plane in 
 $\trop \mathbb P^{n}$. 
This   is the closure  in $\trop \mathbb P^{n}$ of the tropicalization of the uniform matroid   of rank $3$ in $\{0,1,\ldots,n\}$, which  is the fan in $\trop T^{n}\cong  \mathbb R^{n+1}/\mathbb R \tf 1$ given by the  $2$-dimensional cones $\pos(\tf e_i,\tf e_j)$ for $0\le i<j\le n$ where $\tf e_0,\ldots,\tf e_{n}$ is the standard basis of $\mathbb R^{n+1}$. 
Let $\Gamma^{\circ}\subseteq \trop (T^{n})$ be the $1$-dimensional fan whose rays are  $\pos(\tf e_{i}+\tf e_{j})$ where $0\le i\neq j\le n$.
The closure of $\Gamma^{\circ}$ in $\trop \mathbb P^{n}$ is a tropical line $\Gamma$ and  since $\Gamma^{\circ}\subseteq \mathcal L\cap \trop T^{n}$ then $\Gamma$ is contained in $\mathcal L$.

Given $p\in\Grp(2,n)$ we denote by $L_p$ the associated plane in $\mathbb P^{n}$.
We show that we can find an open semi-algebraic set $\mathcal U$ in $\Grp(2,n)$ such that for every $p\in\mathcal U$ we have  $\trop L_p=\mathcal L$ and  there does not exist $\ell\subseteq  L_p$ such that $\trop \ell=\Gamma$.

Firstly we have that $\trop L_p=\mathcal L$ if and only if $p\in \mathcal U_1$ where $$\mathcal U_1=\{q\in \Grp(2,n): \val(q)=(0,\ldots,0)\}.$$

The plane $L_p$ induces a  line arrangement $\mathcal A=\{\ell_0,\ldots,\ell_{n}\}\subseteq \mathbb P^{n}$ given by the lines  $\ell_i=L_p\cap \{x_i=0\}$, with  $x_0,\ldots,x_{n}$   coordinates of $\mathbb P^{n}$. Let $i,j$ be two distinct indices then we denote by $w_{i,j}$  the point of intersection of $\ell_i$ and $\ell_j$. 
There exists a Zariski open set $\mathcal U_2$  of $\Grp(2,n)$ such that for every $p\in \mathcal V$ the line arrangement induced by $L_p$ satisfies  the following conditions
\begin{itemize}
\item[(I)] $\ell_i\cap \ell_j\cap \ell_k=\emptyset$ for any three distinct indices $i,j,k$;
\item[(II)] $w_{i_0,i_1},w_{i_2,i_3},w_{i_4,i_5}$  are not collinear unless $\{i_0,i_1\}\cap \{i_2,i_3\}\cap \{i_4,i_5\}\neq \emptyset$.
\end{itemize}

Let  $\mathcal U$ be the set  $\mathcal U_1\cap\mathcal U_2$. We  prove that if $p\in \mathcal U$ then $\Gamma$ is not realisable in $L_p$.

Suppose there exists a line $\ell\subseteq L_p$ such that  $\trop \ell=\Gamma$. Let 
$ O_{i,j}$ be the orbit of $\mathbb P^n$ where $x_i=x_j=0$,  then by  Theorem 6.3.4 in \cite{M-S} we have that $\ell\cap O_{k,k+1}\neq \emptyset$  for $k=0,\ldots,n-1$ if $n$ is odd and  for $k=0,\ldots,n-2$ if $n$ is even. In fact we have that $\trop \ell\cap \pos(\tf e_{k},\tf e_{k+1}) =\Gamma\cap \pos(\tf e_{k},\tf e_{k+1})=\pos(\tf e_{k}+\tf e_{k+1})$. Moreover  $\ell\cap O_{k,k+1}\subseteq L_p\cap O_{k,k+1}=\ell_{k}\cap \ell_{k+1}=w_{k,k+1}$ hence $\ell\cap O_{k,k+1}=w_{k,k+1}$. This implies that  $w_{0,1},\ldots,w_{n-1,n}$ (\emph{resp.} $w_{0,1},\ldots,w_{n-2,n-1}$ )  are collinear  and if $n\ge 5$ this is a contradiction since $L_p$ satisfies condition (II).

\end{proof}

\begin{remark}
Note that condition (I) is satisfied by all linear spaces $L_p$ with $p\in\mathcal U$. In fact 
$\ell_i\cap\ell_j\cap\ell_k=\emptyset$ if and only if $\trop \ell_i\cap\trop \ell_j\cap\trop \ell_k=\emptyset$.  Since  $\trop L_p=\mathcal L$ we have that $\trop\ell_i=\trop L_p\cap \{x_i=\infty\}=\mathcal L\cap \{x_i=\infty\}$ and by definition of $\mathcal L$ the intersection $\trop \ell_i\cap\trop \ell_j\cap\trop \ell_k$ is empty for every triple of distinct indices $i,j,k$.
%the matroid associated to $L_p$ is the uniform matroid 
\end{remark}

In the following examples we will always assume $\Bbbk$ to be the field of generalised Puiseux series  $\mathbb C((\mathbb R))$ with the natural valuation associated to it (see \cite[Example 2.17]{M-S}). The explicit computations for the  tropical varieties and prevarieties
are done with Tropical.m2 \cite{Trop2}, while we use \verb|Polymake| \cite{GJ00} and the \textit{Polyhedra} package in \verb|Macaulay2| \cite{M2}
 to get the tree associated the tropical lines  in a cone of $\F_1(\trop L)$.

\begin{es}\label{saras}
Let  $L$ be the plane spanned by the rows of the following matrix 
$$
\begin{pmatrix}
0& -271& -92& 0& -13& -54\\
0& -18& -7& -1& 0& -4\\
-1& 12293 & 4173 & 0 & 588& 2450	
\end{pmatrix}.
$$

The  line arrangement  $\mathcal A=\{\ell_i=L\cap \{x_i=0\} :i=0,\ldots,5\}$  satisfies conditions (I) and (II) in the proof of Theorem  \ref{counterexample}. 
The coordinates of  the point $p\in \Grp(2,5)$   associated to $L$ are non zero complex numbers hence $\val(p)=(0,\ldots,0)$. This implies that $\trop L=\mathcal L$ hence   $p\in \mathcal U$.
The Fano scheme $\F_1(L)$ is defined by the ideal 
\begin{eqnarray*}
&&(49p_{25}-37p_{35}-29p_{45},49p_{15}+40p_{35}-64p_{45},49
      p_{05}-26p_{35}-27p_{45},\\
&&98p_{24}-74p_{34}+153p_{45},98p_{14}+80p_{34}+461p_{45},\\
&&98p_{04}-52p_{34}-13p_{45},98p_{23}+58p
      _{34}+153p_{35},\\
&&98p_{13}+128p_{34}+461p_{35},98p_{03}+54p_{34}-13p_{35},\\
&&98p_{12}+144p_{34}+473p_{35}+73p_{45},98p_{02}+10p
      _{34}-91p_{35}-92p_{45},\\
&&98p_{01}-112p_{34}-234p_{35}-271p_{45})
\end{eqnarray*}

The tropicalization $\trop F_1(L)$ is $2-$dimensional fan in $\trop \mathbb  P^{9}$.

The tropical Fano scheme  $\F_1(\trop L)$ is the tropical prevariety defined by the tropical incidence relations associated to $\trop L$ (\cite[Theorem 1]{Haque}). 
These are given by the Pl\"ucker relations generating $\Grp(1,5)$ and by  all tropical polynomials of the form  $$\bigoplus_{i\in T\setminus S} p_{S\cup i}p_{T\setminus i} $$ where $S\subseteq \{0,1,2,3\}=T$, $|S|=1$ and $p_{T\setminus i}$ are the valuations of  coordinates of $p$. In this case $p_{T\setminus i}=0$ for all $0\le i\le3$.

Computations  show that while $\trop F_1(L)\cap \trop T^{9} $ is a $2$-dimensional fan, the tropical Fano scheme $\F_1(\trop L)\cap  \trop T^{9} $ is a fan with $15$ maximal cones of dimension $3$ and $30$ maximal cones of   dimension $2$. The rays of $\F_1(\trop L) \cap \trop T^{9}$ are the same as the rays of $\trop F_1(L) \cap \trop T^{9}$ and the dimension $2$ maximal cones are also cones of $\trop F_1(L) \cap \trop T^{9}$.
The dimension $3$ cones of $\F_1(\trop L)\cap  \trop T^{9}$ are the ones parametrizing tropical lines  whose \textit{combinatorial type} (see Section \ref{proofThm} for a definition) is a snow-flake tree. This is the graph in    Figure \ref{snowflakeType} whose leaves are labelled by numbers from $0$ to $5$. The $2$-dimensional faces of these cones are contained in $\trop \F_1(L)$. The relative interior is parametrising all tropical lines not realisable in $L$. In  Figure \ref{snowflake} we have an example of one of these tropical lines.
\end{es}

\begin{figure}
\definecolor{rvwvcq}{rgb}{0.30196078431372547,0.30196078431372547,2}\begin{tikzpicture}[line cap=round,line join=round,>=triangle 45,x=0.5cm,y=0.5cm]\clip(-15.64,-6.91) rectangle (12.04,7.31);\draw [line width=2pt] (-4,3)-- (-2,1);\draw [line width=2pt] (-2,1)-- (0,3);\draw [line width=2pt] (0,3)-- (0,5);\draw [line width=2pt] (0,3)-- (2,3);\draw [line width=2pt] (-4,3)-- (-4,5);\draw [line width=2pt] (-4,3)-- (-6,3);\draw [line width=2pt] (-2,1)-- (-2,-2);\draw [line width=2pt] (-2,-2)-- (-4,-4);\draw [line width=2pt] (-2,-2)-- (0,-4);\begin{scriptsize}
\draw [fill=rvwvcq] (0,5) circle (2.5pt);\draw[color=rvwvcq] (0.16,5.44) node {};

\draw [fill=rvwvcq] (2,3) circle (2.5pt);\draw[color=rvwvcq] (2.16,3.44) node {};\draw [fill=rvwvcq] (-4,5) circle (2.5pt);\draw[color=rvwvcq] (-3.84,5.44) node {};\draw [fill=rvwvcq] (-6,3) circle (2.5pt);\draw[color=rvwvcq] (-6,3.44) node {};
\draw [fill=rvwvcq] (-4,-4) circle (2.5pt);\draw[color=rvwvcq] (-4.2,-3.58) node {};\draw [fill=rvwvcq] (0,-4) circle (2.5pt);\draw[color=rvwvcq] (0.16,-3.58) node {};\end{scriptsize}
\end{tikzpicture}
 \caption{A snow-flake tree in $\Grp(1,5)$. }
   \label{snowflakeType} 
\end{figure}

In the next example  we show that it is possible to realise the line $\Gamma$ in the proof of Theorem \ref{counterexample} by choosing  a particular $L'$ with $\trop L'=\mathcal L$. 

\begin{es}\label{L'}
Let $L'\subseteq \mathbb P^5$ be the plane spanned by the rows of the following matrix:
$$
\begin{pmatrix}
1& 3& 0& 1& 5& 7\\
0& 0 & 1 & 3& -1&-1\\
1 & 4 &-1 &-3 & 0 & 0
\end{pmatrix}
$$

The line arrangement $\mathcal A'$ associated to $L'$ satisfies condition (I) of the proof of Theorem \ref{counterexample} and  we have  $\trop L'=\mathcal L=\trop L$. However $\mathcal A'$ does not satisfy condition (II). 
Let $p'_{i,j}$ be the point $L'\cap O_{i,j}$.  The points $p'_{0,1},p'_{2,3}$ and $p'_{4,5}$ are collinear and  the line $\ell$  passing through them is defined by  the following equations
\begin{eqnarray*}
x_4-x_5=0, 3x_2-x_3=0,3x_1+4x_3+12x_5=0,3x_0+x_3+3x_5=0.
\end{eqnarray*}

The tropical line $\trop \ell $ is the closure in $\trop \mathbb P^5$ of the fan in $\trop T^5$ whose  rays are  $\pos(\tf e_0+\tf e_1),\pos(\tf e_2+\tf e_3),\pos(\tf e_4+\tf e_5)$. Hence this is the tropical line $\Gamma$ of the proof of Theorem \ref{counterexample}. We now compare $\trop \F_1(L')$ with  $\trop \F_1(L)$.
The ideal associated to the Fano scheme $F_1(L')$ is
\begin{eqnarray*}
&&(6p_{25}-2p_{35}-p_{45},6p_{15}+8p_{35}+97p_{45},6p_{05}+2p_{35}+25p_{45}\\
&&6p_{24}-2p_{34}-p_{45},6p_{14}+8p_{34}+73p_{45},6p_{04}+2p_{34}+19p_{45}\\
&&6p_{23}+p_{34}-p_{35},6p_{1,3 }-97p_{34}+73p_{35},\\
&&6p_{03}-25p_{34}+19p_{35},6p_{12}-31p_{34}+23p_{35}-4p_{45},\\
&&6p_{02}-8p_{34}+6p_{35}-p_{45},6p_{01}+p_{34}-p_{35}-3p_{45})
\end{eqnarray*}
and $\trop \F_1(L')$ is a $2-$dimensional fan in $\trop \mathbb P^5$. 
Let $L$ be the plane of Example \ref{saras}. Since $\trop L=\trop L'$ then $F_1(\trop L)=\F_1(\trop L')$ and both  $\trop \F_1(L')$ and  $\trop \F_1(L)$ are contained in $\F_1(\trop L)$. All rays of   $\trop \F_1(L)$ are also rays of $\trop \F_1(L')$ but $\trop \F_1(L')$ has also an extra ray $r$ that is not contained in  $\trop \F_1(L)$.   
The combinatorial type of  the tropical lines associated to points in $r$  is the snowflake in  Figure \ref{snowflakeType}. 
Moreover $r$ is the barycentre of the  $3$-dimensional cone $C$ of $\F_1(\trop L)$ containing $r$ in its relative interior. If $C=\pos(r_1,r_2,r_3)$ then $r=\pos(r_1+r_2+r_3)$. We have that $C\cap \trop F_1(L)$ is given by the two dimensional faces of $C$. On the other hand $C\cap \trop F_1(L') $ is the union of the three cones $\pos(r_1,r_1+r_2+r_3),\pos(r_2,r_1+r_2+r_3),\pos(r_3,r_1+r_2+r_3)$ (see Figure \ref{triang}).

\begin{figure}
\begin{center}
    \begin{tikzpicture}[scale=1]
\begin{scope}
\draw(-1,0)--(7.3,0);
\draw (3,0)--(0.5,4.8);

\draw(3,0)--(5.3,4.8);

\draw (3,0)--(7.3,-4);
\draw (-1,-4)--(3,0);

\node [right] at (5,4){$\tf e_1$};
\node [right] at (0.2,4){$\tf e_0$};
\node [right] at (6.7,0.2){$\tf e_2$};
\node [right] at (6.8,-3.5){$\tf e_3$};
\node [right] at (-1.2,-3.5){$\tf e_4$};
\node [right] at (-0.9,0.2){$\tf e_5$};

%ramo 23
\draw[-][ultra thick ,red]  (3,0)--(5.5,-1);
\draw[-][ultra thick ,red]  (5.5,-1)--(6.8,-1);
\draw[-][ultra thick ,red]  (5.5,-1)--(6.8,-2);

%ramo 01
\draw [-][ultra thick ,red] (3,0)--(3,2.5);
\draw[-][ultra thick ,red]  (3,2.5)--(4,4.6);
\draw[-][ultra thick ,red]  (3,2.5)--(2,4.6);

%ramo 45
\draw [-][ultra thick ,red]  (3,0)--(0.5,-1);
\draw[-][ultra thick ,red]  (0.5,-1)--(-0.8,-1);
\draw[-][ultra thick ,red]  (0.5,-1)--(-0.8,-2);

%\draw[fill= green] (6.3,1.5) -- (6.3,1) -- (5.3,0.6) -- (5.3,2)--cycle;

%\draw [fill] (6.3,1) circle [radius=0.08];
%\draw [fill] (5.3,0.6) circle [radius=0.08];
%\draw [fill] (5.3,2) circle [radius=0.08];

\end{scope}
\end{tikzpicture}
\end{center}
    \caption{A tropical line  contained in the three cones $\pos(\tf e_0,\tf e_3),\pos(\tf e_2,\tf e_5),\pos(\tf e_1,\tf e_4)$ of $\trop L\subseteq \mathbb R^5\cong \mathbb R^6/\mathbb R \tf 1 $ as in Example \ref{saras}.}
   \label{snowflake} 
\end{figure}
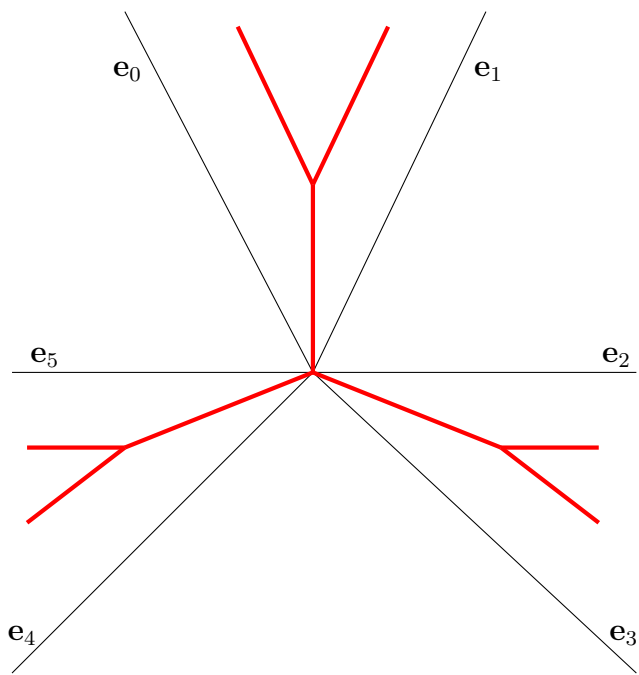

\end{es}

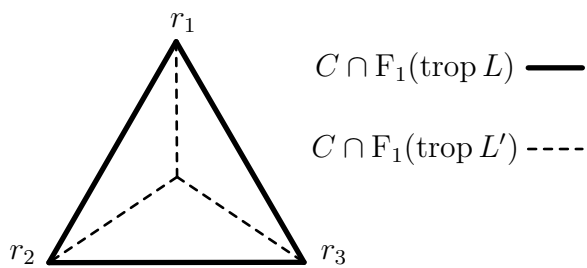
\begin{figure}
\begin{center}
\definecolor{ccqqqq}{rgb}{0.8,0,0}\definecolor{uuuuuu}{rgb}{0.26666666666666666,0.26666666666666666,0.26666666666666666}\definecolor{ududff}{rgb}{0.30196078431372547,0.30196078431372547,1}\begin{tikzpicture}[line cap=round,line join=round,>=triangle 45,x=0.6cm,y=0.6cm]\clip(-13,0) rectangle (4,6);\draw [line width=2pt] (-5.3,5.08)-- (-8.12,0.18);\draw [line width=2pt] (-8.12,0.18)-- (-2.4664755214562497,0.18780836132788226);\draw [line width=2pt] (-2.4664755214562497,0.18780836132788226)-- (-5.3,5.08);
\draw [line width=1pt,dashed] (-5.28,2.08)-- (-8.12,0.18);
\draw [line width=1pt,dashed] (-5.28,2.08)-- (-5.3,5.08);
\draw [line width=1pt,dashed] (-5.28,2.08)-- (-2.4664755214562497,0.18780836132788226);\draw (-5.14,5.51) node {$r_1$};\draw (-8.7,0.35) node {$r_2$};\draw (-1.8,0.35) node {$r_3$};
\draw (0,4.5) node {$C\cap \F_1(\trop L) $};\draw [line width=2pt](2.5,4.5)--(3.7,4.5);
\draw (0,2.7) node {$C\cap\F_1(\trop L')$};
\draw [line width=1pt,dashed](2.5,2.7)--(3.7,2.7);
%\draw(-5.12,2.51) node {$r_1+r_2+r_3$};
\end{tikzpicture}
\end{center}
\caption{ A section of the cone $C\subseteq \F_1(\trop L)$ as in Example \ref{L'}. }
\label{triang}
\end{figure}

In Example \ref{nonfan} we exhibit a plane $L''$ such that $\trop (L''\cap T^n)$ is a fan but $\trop (\F_1(L'')\cap T^{\binom{n+1}{2}-1})$ is not. This shows that Proposition \ref{fanProperty} does not hold if we replace $\F_1(\trop X)$ with $\trop\F_1(X)$.

\begin{es}\label{nonfan}
 Let $L''$ be the plane in $ \mathbb P^5$  spanned by the rows of the following matrix 
$$
M=\begin{pmatrix}
1&1&0&t&1&1\\
1&t+1&1&2&t&0\\
5&8&6&9&7&10
\end{pmatrix}.
$$

We have that $\trop L''=\trop L$ with $L$ the plane in Example \ref{saras} and 
the line arrangement $\mathcal A''=\{ L''\cap O_{i,j}:0\le i<j\le 5\}$ satisfies condition (I) of proof of Theorem \ref{counterexample}. Moreover the points $p''_{01}=L''\cap O_{0,1},p''_{23}=L''\cap O_{2,3}$ and $p''_{45}=L''\cap O_{4,5}$ are not collinear.

The line spanned by the first two rows of $M$ tropicalizes to a tropical line  whose combinatorial type is a snowflake tree  whose pairs of leaves  are labelled by $i$ and $i+1$ for $i=0,..,4$. The corresponding point in $\trop F_1( L'')$ is $\tf e_{01}+\tf e_{23}+\tf e_{45}$ in $\mathcal O=\trop (\Grp(1,5)\cap T^{\binom{6}{2}})\subseteq \mathbb R^{\binom{6}{2}}/\mathbb R \tf 1$, where the $\tf e_{ij}$'s denote the  standard basis vectors of  $\mathbb R^{\binom{6}{2}}$.

We want to show that $\trop\F_1(L'')$ is not a fan by proving that  the ray $\pos(\tf e_{01}+\tf e_{23}+\tf e_{45})$ is not  contained in $\trop (\F_1( L'')\cap T^{\binom{6}{2}})$. 

By contradiction suppose $\pos(\tf e_{01}+\tf e_{23}+\tf e_{45})\subseteq \trop (\F_1( L'')\cap T^{\binom{6}{2}})$ then  its closure  in $\trop \mathbb P^{\binom{6}{2}}$ is a point $Q$ and it  is contained in $\trop \F_1(L'')$. 

The  point  $Q$ is  in the orbit  $\mathcal O=\{[p_{ij}]\in \trop  \mathbb P^{\binom{6}{2}-1}: p_{01}=p_{23}=p_{45}=\infty \}$ and $Q_{ij}=0$ for $ij\neq 01,23,45$.  The tropical line  $\Gamma_Q$ is  given by the fan in $\trop\mathbb P^5$ with rays $\pos(\tf e_0+\tf e_1),\pos(\tf e_2+\tf e_3)$ and $\pos(\tf e_4+\tf e_5)$. Moreover  $\Gamma_Q$  is not realizable in $L''$ otherwise the points   $p''_{01},p''_{23}$ and $p''_{45}$ would be  collinear. 
\end{es}

Another instance where the containment $\trop \F_1(X)\subseteq \F_1(\trop X)$ is strict is the case of \textit{general} hypersurfaces whose tropicalization has the same support of a tropical linear space. An hypersurface is \textit{general}  if its Fano scheme of lines has dimension $2n-d-3$ (see  \cite[Theorem 8]{barth}).

\begin{prop}\label{LinearObstruction}
If $X$ is a general hypersurface of degree $d>1$  and the tropicalization $\trop X$ has the same support as a  tropical linear space then  $\trop(\F_1(X))\subsetneq \F_1(\trop X).$
\end{prop}
\begin{proof}
If  $L$ is a $(n-1)$-dimensional linear space then  the dimension of $\F_1(L)$  is $\dim \Grp(2,n)=2n-4$. By hypothesis we have that $\F_1(\trop X)=\F_1(\trop L)$ and $\dim  \F_1(\trop L)\ge  \dim \trop F_1(L)=2n-4$. 
On the other hand the dimension of $\trop F_1(X)$ is equal to the dimension of $\F_1(X)$  which is $2n-d-3$. Suppose $\trop F_1(X)=\F_1(\trop X)$ then 
we would have   $2n-d-3 \ge  2n-4$ but this is not the case if $d>1$.
\end{proof}

\section{Toric varieties }\label{toricv}
In this section we look at Fano schemes of  toric varieties. We  prove that  for these varieties  the tropical Fano scheme is equal to the tropicalization of the classical Fano scheme.\\

Consider a toric variety $X$  associated to a set of lattice points $\mathcal A=\{\tf a_0,\ldots,\tf a_n\}$ with $\mathcal A \subseteq \ZZ^m \times \{1\}  $ and denote by  $A$ the matrix whose columns are the points in $\mathcal A$.
The variety $X$ has a natural   embedding  in $\mathbb P^n$  given by a monomial map
 $\phi_{\mathcal A}:(\Bbbk^*)^m\times \Bbbk^*\to \mathbb P^n$ (see \cite[Section 2.1]{Cox2}). 
   We denote  the closure of the image of this map by $X_{\mathcal A}$.  The matrix $A$  also defines a map $\trop (\phi_{\mathcal A}):\mathbb {R}^{m+1}\to \mathbb {R}^{n+1}$. By  \cite[Theorem 3.2.13]{M-S} we have that  $\trop(X_{\mathcal A}\cap T^n)\subseteq \mathbb R^{n+1}/\mathbb R \tf 1$ is the quotient by $\mathbb R \tf 1 $ of the  image of $\trop (\phi_{\mathcal A})$ which  is  the  classical linear space spanned by the rows of $A$. Since the embedding of the toric variety only depends on the row span of $A$ (\cite[Proposition 1.1.9]{Cox2}) it is possible to  recover the ideal defining $X_{\mathcal A}$   from  $\trop (X_{\mathcal A}\cap T^n)$.

\begin{es}\label{firstEs}
 Let $X_{\mathcal A}\subseteq \mathbb P^3$ be the toric variety associated to the set of lattice points $\mathcal A=\{(1,1,1),(0,0,1),(0,-1,1),(1,0,1)\}$. The matrix $A$ is $$\begin{pmatrix}
1&0&0&1\\
1&0&-1&0\\
 1&1&1&1
 \end{pmatrix} 
 $$ and the ideal defining $X_{\mathcal A}$ is $(xz-yw)$.  The tropicalization  $\trop (X_{\mathcal A} \cap T^3)$ is the quotient by $\mathbb R \tf 1$ of  $\{(x,y,z,w):x+z=y+w\}$ and this   is equal to the quotient by $\mathbb R \tf 1$ of the linear span of the rows of A.
\end{es}

By contrast with the case of linear spaces we show that for  toric varieties   the tropical Fano scheme is the same as the tropicalization of the classical Fano scheme. 

\begin{thm}\label{equality}
Let $X=X_{\mathcal A}$ be a toric variety. Then $\F_d(\trop X)=\trop \F_d(X)$.
\end{thm}

We prove  this result by showing that  for each tropicalized linear space $\Gamma\subseteq \trop X$ there exists a linear space $\ell\subseteq X$ that tropicalizes to it. 
We explicitly  construct  $\ell$ using  \textit{Cayley structures}  on $\mathcal A$. We use results  in  \cite[Section 3]{I-Z} where the authors prove that for  each $s-$Cayley structure $\pi$
there exists a  subvariety $Z_{\pi}$ of  $ \F_{s}(X_{\mathcal A})$ and from  $\pi$ it is also possible to   deduce equations of the linear spaces parametrised by $Z_{\pi}$. \\

Given a set of $n+1$ lattice points $\mathcal A$ in  $\mathbb Z^m\times \{1\}$, let $L$ be the kernel of the map defined by the matrix $A$ and $\tf e_i$ be the standard basis vectors of $\mathbb R^{n+1}$. If $\tf l\in L$ we can write $\tf l=\sum_{l_i>0} l_i\tf e_i-\sum_{l_i<0} -l_i\tf e_i$ and denote  by $l^+=\sum_{l_i>0} l_i\tf e_i$ and $l^-=\sum_{-l_i<0}- l_i\tf e_i$. We have that $\tf l\in L$ if and only if $ \sum_i l_i \tf a_i=0$.
 The toric variety $X_{\mathcal A}\subseteq \mathbb P^n$ is generated by binomials of the form $\tf x^{l^+}-\tf x^{l^-}=\prod_{l_i>0} x_i^{l_i} -\prod_{l_i<0} x_i^{l_i} $  with $\tf l\in L$ (\cite[Proposition 1.1.9]{M-S}). 
 
A face $\tau $ of $\mathcal A$ is the intersection of a face of $\conv (\mathcal A)$ with $\mathcal A$. Denote by $\Delta_s$ the standard basis $\{\tf e_0,\ldots,\tf e_s\}$ of $\mathbb Z^{s+1}$. 
\begin{deff}
An $s$-Cayley structure on $\tau$  is  a surjective map $\pi : \tau \to  \Delta_s$
such that if
$\tf l\in L,l_i\neq 0$ for all $i$ with  $\tf a_i\in\tau$   and $\sum_{l_i\neq 0} l_i\tf a_i=0$ then 
$\sum_{ l_i\neq 0}  l_i\pi(\tf a_i)=0$, or equivalently  
$\sum_{l_i>0}l_i \pi(\tf a_i)=\sum_{l_i<0}-l_i \pi(\tf a_i)$.
\end{deff}

\begin{es}
Consider the set of lattice points $\mathcal A$ as in Example \ref{firstEs}.
A $1-$Cayley structure  is given by $\pi:\mathcal A\to \mathbb Z^2$ with $\pi((0,0,1))=\pi((0,-1,1))=\tf e_0$ and $\pi((1,0,1))=\pi((1,1,1))=\tf e_1.$
An example of a surjective map $\pi:\mathcal A\to \Delta_1$ that is not a Cayley structure  is given by $\pi:\mathcal A\to \mathbb Z^2$ with $\pi((1,1,1))=\pi((0,-1,1))=\tf e_0$ and $\pi((0,0,1))=\pi((1,0,1))=\tf e_1$. We can see that $\tf l=(1,-1,1,-1)$ is in $L$  hence $(1,1,1)-(0,0,1)+(0,-1,1)-(1,0,1)=0$ but if we apply $\pi$ we get $2\tf e_1-2\tf e_2=0$ which is a contradiction. 
\end{es}

We now prove that given a tropicalized  linear space  in $\trop (X\cap T^n)$ we can associate a Cayley structure on $\mathcal A$ to it.

Let $\Gamma$ be a $d$-dimensional  tropicalized linear space in $\trop T^n$ and let $M_{\Gamma}$ be the matroid associated to it. This is the matroid on $\{0,1,\ldots n\}$ whose bases are the set $\{i_0,\ldots,i_d\}$ such that the corresponding Pl\"ucker coordinates $p_{i_0,\ldots,i_d}$ is not zero. Note that this matrix does not have loops, circuits of one element.\\
The recession fan of $\Gamma$ is the  fan  whose cones are $\pos(\textbf {e}_{F_1},\ldots, \textbf {e}_{F_{d+1}})+\mathbb R  \textbf {1}$ where $\emptyset\neq F_1\subsetneq  \ldots\subsetneq  F_{d+1}$ is a
 maximal chain of flats of $M_{\Gamma}$, $\textbf e_{F_i}=\sum_{j\in F_i}\textbf {e}_i$ 
 and $(\textbf{e}_i)_k=1 $ for $k=i$ and $(\textbf{e}_i)_k=0$ otherwise.

\begin{prop}\label{Cay}
Let $X_{\mathcal A}\subseteq \mathbb P^n$ be a toric variety and let $\Gamma$ be a tropicalized linear space contained in $\trop (X_{\mathcal A}\cap T^n)$.
If  $M_{\Gamma}$ has $m+1$ non-empty  minimal flats   then there exists an  $m-$Cayley structure on $\mathcal A$. 
\end{prop}

The following is a technical lemma which will be used for the proof of Proposition \ref{Cay}.

\begin{lemma}\label{lines}
Let $\Gamma \subseteq \trop T^n$ be  a tropicalized  linear space  and $\{F^0_1,\ldots, F^m_1\}$ the set of non-empty minimal flats of
$M_{\Gamma}$. Then 
\begin{itemize}
\item [(i)]  there exists a unique $F_1^j$ such that $i\in F_1^j$;
\item [(ii)]   $\bigcup _{j=1}^m F^j_1=\{0,\ldots,n\}$.
\end{itemize}

\end{lemma}

\begin{proof}
For (i) we observe that if $i\in F_1^j\cap F_1^k$ then, since there are no loops, $\{i\}$ would  also be  a flat but this would contradict the minimality of $F_1^j$ and $ F_1^k$.

If there exists $i\in \{0,\ldots,n\}$ that is not in $\bigcup _{j=1}^m F^j_1$ then $\{i\}$ can not be  a flat. This implies that  it is a loop but this is a contradiction since $M_{\Gamma}$ has no loops.
\end{proof}

\begin{proof}[Proof of Proposition \ref{Cay}]

Let $\Gamma$  be a tropicalized  linear space contained in $\trop (X\cap T^n)$ and $F^0_1,\ldots, F^m_1$ the non-empty minimal flats of $M_{\Gamma}.$ 
The ray $\pos(\textbf {e}_{F^i_1}) $ of  $\Gamma$ is contained in $\trop ( X\cap T^n)$ for all $i$  hence the    vectors $\textbf {e}_{F^0_1},\ldots,\textbf {e}_{F^m_1}$   are part of a set of generators for the linear space $\trop( X\cap T^n)$.  Lemma \ref{lines}   implies that  they are linearly independent vectors in $\mathbb R^{n+1}$. 
The linear span in $\mathbb R^{n+1}/\mathbb R\textbf 1$ of $\textbf {e}_{F^0_1},\ldots,\textbf {e}_{F^m_1}$  is equal to the linear span of  $\textbf {e}_{F^1_1},\ldots,\textbf {e}_{F^m_1}$ and $(1,\ldots,1)$. Hence we can assume that $\textbf {e}_{F^1_1},\ldots,\textbf {e }_{F^m_1}$ are the first $m$ rows of $A$ and $\textbf {e}_{F^0_1}$ is the unique among 
$\textbf {e}_{F^0_1},\ldots,\textbf {e}_{F^m_1}$
 with last coordinate equal to $1$.
The columns of $A$ are the points of $\mathcal A$ and by Lemma \ref{lines} they  can be  partitioned in $m+1$ sets $A_0,\ldots,A_{m}$. The set $A_i$, for $i=0,\ldots,m-1$, is given all  points whose coordinates $(p_0,\ldots,p_n)$ are such that $p_i=1$ and $p_j=0$ for all $0\le j\neq i\le m$. The set  $A_{m}$ is given by the points whose first $m$ coordinates are zero. We have that  $A_0\cup \ldots\cup A_m=\mathcal A$. In fact 
by  Lemma \ref{lines} for any $i$ there exists a unique $\tf e_{F_1^j}$ such that $(\tf e_{F_1^j})_i=1$. This implies for  each point $(p_0,\ldots,p_n)$ in $\mathcal A$ (equivalently each column of   $A$) there exists a unique $0\le i\le m$ such  that  $p_i=1$.  Since each $\tf {e}_{F_1^j}$ has at least one coordinate equal to $1$ we have that $A_0\cup\ldots\cup A_{m-1}\subseteq \mathcal A$. Moreover since the first $m$ rows of $A$ are $\textbf {e}_{F^1_1},\ldots,\textbf {e}_{F^m_1}$  we have that the last column of $A$ has first $m$ entries equal to zero. Hence $A_{m}\neq \emptyset$ and $\mathcal A=\cup _{i=0}^s A_i$. 
We define   $\pi:\mathcal A \to \Delta_m$ to be the map that sends  the points in $A_r$ to  $\tf e_{r+1}\in\mathbb Z^{m+1}$.  This map is an $m-$Cayley structure on $\mathcal A$. 
In fact let $\tf l\in L$ with $\tf l=l^{+}-l^{-}=\sum_{l_i>0}l_i\tf e_i-\sum_{l_i<0}l_i\tf e_i$ and $\{i:l_i\neq 0\}=\{i:\tf a_i\in \mathcal A\}$ then we have 
$$
\sum_{l_i>0,\tf a_i\in A_0} l_i \tf a_i+\ldots+\sum_{l_i>0,\tf a_i\in A_m} l_i \tf a_i=\sum_{l_i<0,\tf a_i\in A_0} -l_i \tf a_i+\ldots+\sum_{l_i<0,\tf a_i\in A_m} -l_i \tf a_i.
$$
We need to prove that 
\begin{equation*}
\sum_{l_i>0,\tf a_i\in A_0} l_i \pi(\tf a_i)+\ldots+\sum_{l_i>0,\tf a_i\in A_m} l_i \pi(\tf a_i)=\sum_{l_i<0,\tf a_i\in A_0} -l_i \pi(\tf a_i)+\ldots+\sum_{l_i<0,\tf a_i\in A_m} -l_i \pi(\tf a_i).
\label{eq1}
\end{equation*}
By definition of $\pi$ we have that
$$
\sum_{l_i>0,\tf a_i\in A_0} l_i \pi(\tf a_i)+\ldots+\sum_{l_i>0,\tf a_i\in A_m} l_i \pi(\tf a_i)=(\sum_{l_i>0,\tf a_i\in A_0} l_i ,\ldots,\sum_{l_i>0,\tf a_i\in A_m} l_i  )
$$
and 
$$
\sum_{l_i<0,\tf a_i\in A_0} -l_i \pi(\tf a_i)+\ldots+\sum_{l_i<0,\tf a_i\in A_m} -l_i \pi(\tf a_s)=(\sum_{l_i<0,\tf a_i\in A_0} -l_i ,\ldots,\sum_{l_i<0,\tf a_i\in A_m}- l_i  ).
$$

Consider  $(p_0,\ldots,p_n)=\sum_{l_i>0,\tf a_i\in A_0} l_i \tf a_i+\ldots+\sum_{l_i>0,\tf a_i\in A_m} l_i \tf a_i=\sum_{l_i<0,\tf a_i\in A_0} -l_i \tf a_i+\ldots+\sum_{l_i<0,\tf a_i\in A_m} -l_i \tf a_i$. 
The first coordinate $p_0$ is given by the first coordinate of $\sum_{l_i>0,\tf a_i\in A_0} l_i \tf a_i$ that is $\sum_{l_i>0,\tf a_i\in A_0} l_i $ or equivalently by  the first coordinate of $\sum_{l_i<0,\tf a_i\in A_0} -l_i \tf a_i$ that is $\sum_{l_i<0,\tf a_i\in A_0} -l_i $.

 From this we obtain  $\sum_{l_i>0,\tf a_i\in A_0} l_i= \sum_{l_i<0,\tf a_i\in A_0} -l_i $. In the same way we have $\sum_{l_i>0,\tf a_i\in A_1} l_i= \sum_{l_i<0,\tf a_i\in A_1} -l_i ,\;\ldots\;,\sum_{l_i>0,\tf a_i\in A_{m-1}} l_i= \sum_{l_i<0,\tf a_i\in A_{m-1}} -l_i$.

Since $p_n=\sum _{l_i>0}l_i=\sum _{l_i<0}-l_i$ we can also deduce that $$\sum_{l_i>0,\tf a_i\in A_m} l_i= \sum_{l_i<0,\tf a_i\in A_sm} -l_i $$ therefore 
 $$
 (\sum_{l_i>0,\tf a_i\in A_0} l_i ,\ldots,\sum_{l_i>0,\tf a_i\in A_m} l_i  )=(\sum_{l_i<0,\tf a_i\in A_0} -l_i ,\ldots,\sum_{l_i<0,\tf a_i\in A_m} -l_i  ).
 $$
\end{proof}

\begin{es}\label{excay}
Let $\mathcal A$ be the set given by the columns of the matrix $A$ where
$$A=\begin{pmatrix}
0&1&0&0&0\\
1&0&0&0&0\\
2&1&7&3&5\\
1&1&1&1&1\\
\end{pmatrix}.$$ 

The toric variety $X_{\mathcal A}$ is defined by the ideal $(x_2x_3-x_4^2)\subseteq \mathbb C[x_0,x_1,x_2,x_3,x_4]$. %\subset \mathbb C[x_0,x_1,x_2,x_3,x_4]$.
The tropical line $\Gamma_1$  spanned by  $(0,1,0,0,0)$  is contained in $\trop (X_{\mathcal A}\cap T^3)$. In the case of tropical lines the cones  $\pos(\textbf {e}_{F^0_1})+\mathbb R\textbf 1,\ldots,\pos(\textbf {e}_{F^m_1})+\mathbb R\textbf 1 $ are exactly the rays of $\Gamma$. We can define a $1-$Cayley structure associated to $\Gamma_1$ by sending the set $$A_0=\{(0,1,2,1),(0,0,7,1),(0,0,3,1),(0,0,5,1)\}$$ 
to $\tf e_0$ and  $A_1=\{(1,0,1,1)\}$  to $\tf e_1$. 

We also notice that the tropical line $\Gamma_2$ whose rays are $\pos(1,0,0,0,0),\pos (0,1,0,0,0)$ and $\pos(-1,-1,0,0,0)$ is contained in $\trop (X_{\mathcal A}\cap T^3)$. The $2-$Cayley structure associated to $\Gamma_1$ is the map sending  $A_0=\{(0,1,2,1)\}$ to $\tf e_0$, $A_1=\{(1,0,1,1)\}$ to $\tf e_1$,  $A_2=\{(0,0,7,1),(0,0,3,1),(0,0,5,1)\}$ to $\tf e_2$.
\end{es}

\begin{proof}[Proof of Theorem \ref{equality}]
We will prove that given a tropicalized  linear space  $\Gamma\subseteq \trop X$ there exists a linear space $\ell '$  in $X$ such that $\trop \ell'=\Gamma$. 

Assume  that $\Gamma$ is in $\trop (X\cap O)$ with $O$ an orbit of $\mathbb P^n$. We can consider $Y=\overline{X\cap O}$ as a subvariety of $\overline { O}\cong \mathbb P^s$ with $s=\dim \overline{O}$. The variety $Y$ is also a toric variety and we denote by $\mathcal A'$ the set of lattice points associated to it.

Suppose $M_{\Gamma}$ has $l+1$ minimal flats. By Lemma \ref{Cay} we have that there exists a $l-$Cayley structure $\pi$  on $\mathcal A'$. 
Let $Z_{\pi}$ be the subvariety of $\F_l(Y)$ associated to $\pi$ (see \cite[Section 3, Section 4]{I-Z}). This is  the closed torus orbit of the  linear space $L$  generated by   $\tf v_0,\ldots,\tf v_l\in \mathbb R^{s+1}$ where 
\begin{equation*}
(\tf v_j)_i=\left \{\begin{matrix} 1\text{ if } \pi(\tf a_j)=\tf e_1\\
0 \text{ else }
\end{matrix}\right. .
\end{equation*} 

Let $\Gamma'$ be  the translation of $\Gamma$ to the origin. There exists  a point $p$  in $\trop Y$ such that  $\Gamma=\Gamma'+p$. 
The    vectors  $\textbf {e}_{F^0_1},\ldots,\textbf {e}_{F^m_1}$   generate  a linear space $\mathcal L$ and $\Gamma\subseteq \mathcal L+p$. We have that $\mathcal L=L$.  In fact by definition of the $(\tf v_j)_i$  and by construction of $\pi$ in Lemma \ref{Cay} the matrix 
\begin{displaymath} \begin{pmatrix}
\tf v_1\\
\vdots\\
\tf v_{l}
\end{pmatrix}
\end{displaymath}
 is equal to the submatrix of $A$ given by the first $l$ rows.
The equations of $L$ are $\codim L $ binomials of type $x_i-x_j$ for pairs $(i,j)$ with $0\le i\neq j\le m$, 
 hence   $\trop L=L$. Moreover there exists $t\in T^{\dim Y} $ such that $\trop (t\cdot L)=\mathcal L+p$. 
 
 We show that $\Gamma$ is the tropicalization of a linear space in $t\cdot L$ hence in $Y$. Using the equations of $L$  we can choose $x_0,\ldots,x_{l} \in\{x_0,\ldots,x_s\}$ such that for any $q\in L$ we have  $q_i=q_j$ for $j\in\{0,\ldots,l\}$. This implies that the projection 
$\phi=\phi_{x_0,\ldots,x_{l} }: \mathbb P^s \to \mathbb P^l$ induces  an isomorphism between $L$ and $\mathbb P^l$.  Let $\psi^{-1} $ be its inverse.  Since $\phi$ and $\phi^{-1}$ are linear monomial maps then $\trop (\phi)=\phi$ and $\trop (\phi^{-1})=\phi^{-1}$. Consider the linear space $\phi(\Gamma')\subseteq \trop \mathbb P^l$. This linear space is realizable in $\mathbb P^l$, that is  there exists $\ell'\in\mathbb P^l$ such that $\trop \ell'=\phi(\Gamma')$.
Now $\phi^{-1}(\ell')\subseteq L\subseteq Y$ and    $\trop(\phi^{-1}(\ell'))=\trop (\phi^{-1})(\trop (\ell'))=\Gamma'$. If we consider $\ell=t\cdot \phi(\ell') $ then $\trop(\ell )=\Gamma$.

\end{proof}

\begin{es}
Consider the toric variety $X_{\mathcal A}$ of Example  \ref{excay}. We use the proof of Theorem \ref{equality} to compute the lines  $\ell_1,\ell_2$ in $X_{\mathcal A}$ that tropicalize  to $\Gamma_1$ and $\Gamma_2$ respectively.
The line $\ell_1$ is the line $L$ associated to the $1-$Cayley structure  $\pi_1$.
Its defining equations are $x_0-x_2=0,x_2-x_3=0,x_3-x_4=0$.
The tropical line $\Gamma_2$ is contained in the linear space $L$ defined by 
$x_2-x_3=0,x_3-x_4=0$. Consider the projection $\phi=\phi_{x_0,x_1,x_2}:\mathbb P^5\to \mathbb P^2$ then $\phi (\Gamma_2) $ is the tropical line in $\mathbb R^3/\mathbb R\tf 1$  with rays $\pos(1,0,0),\pos(0,1,0),\pos(0,0,1)$ and it is the tropicalization of the line  $V(x_0+x_1+x_2)$. Applying $\phi$ we get that $\ell_2$ is defined by $(x_0+x_1+x_2,x_2x_3-x_4^2,x_3-x_4)$. 
\end{es}

\section{Proof of Theorem \ref{Fanstructure} and Proposition \ref{fanProperty}}\label{proofThm}

In this section we prove Theorem \ref{Fanstructure} by showing that there exists  a polyhedral structure on 
each $\F_d(\trop X)\cap \mathcal O$.

The key point in the proof of Theorem \ref{Fanstructure} is  the identification of $\trop \Grp(d,n)\cap \mathcal O$ with the subfan of the secondary fan $\Sigma$ of the matroid polytope  $P_M$ (\cite[Definition 4.2.9 ]{M-S}). We see in the following paragraph that  $M$ is the uniform matroid associated to $\trop \Grp(d,n)\cap \mathcal O$.  The cones of this subfan are the intersection of $\trop \Grp(d,n)\cap \mathcal O$ with the cones of $\Sigma$ and the subdivisions associated to these cones are  the \textit{matroid subdivisions} (see \cite[\S 4.4 ]{M-S} for a definition).

The space $\trop \Grp(d,n)\cap  \trop T^{\binom{n+1}{d+1}-1}$ was first studied by Speyer and Sturmfels in \cite{Speyer} and  can be identified with a subfan of the secondary fan of the uniform matroid of rank $d+1$ on $\{0,1,\ldots,n\}$\cite[\S4.4]{M-S}. The same interpretation of  $\trop \Grp(d,n)\cap \mathcal O$ can be extended to the case  where  $\mathcal O$ is any orbit of
$ \trop \mathbb P^{\binom{n+1}{d+1}-1}$. This is done in
 the forthcoming paper of  Cueto and Corey \cite{CC}. 
In particular they show that 
$$\Grp(d,n)\cap \overline O\cong\Grp(1,n')\times  \prod_{j\in J} T^j$$ where $n'<n$ and $J\subseteq \mathbb N$ with $|J|<\infty$. The isomorphism between them is a  map $\psi=\pi\times f$ where $\pi$ is a projection and $f$ is a monomial map. Hence it is possible to consider the tropicalization of this map to get 
$$\trop \Grp(d,n)\cap \overline{ \mathcal O}\cong\trop \Grp(d,n')\times  \prod_{j\in J} \trop  T^j.$$
Let  $M'$ be the uniform matroid   of rank $d+1$ on $\{0,1,\ldots,n'\}$.  We can identify   $\trop \Grp(d,n)\cap \mathcal O$ with a product of a subfan of the secondary fan of $P_{M'}$ with  $ \prod_{j\in J} \trop  T^j=\mathbb R^{\sum_{j\in J} j}/\mathbb R\textbf 1$.

This identification induces a polyhedral structure on 
$\trop \Grp (d,n)$   given by the union of cones $C_{T}$ where each $T$ is a different matroid subdivision.
Consider $p$  in the relative interior   $C_T^{\circ}$ of $  C_T$ and the corresponding tropical linear space $\Gamma_p$. We say that  the \textit{combinatorial type} of $\Gamma_p$ is   $T$. If  $p$  is contained in $\in C_T\setminus C_T^{\circ}$ then
the combinatorial type of $\Gamma_p$ is $T'$ where $T'$ is the matroid subdivision associated to a cone $C_{T'}$ in the boundary of $C_T$ such that  $p\in C_{T'}^{\circ}$.  Note that if  $C_{T'}$ is in the boundary of $C_T$ then a cell  $\sigma$ in $C_{T}$ is either equal to a cell in $C_{T'}$ or it is obtained by subdividing a cell $\sigma'$ of $C_{T'}$. In the second case all the cells in the subdivision of $\sigma'$ are cells of in $C_T$.
%From this it follows that the some cells in the tropical linear spaces associated to $C_{T}$
%--to be finished...
%--questo signifca che le celle del trop linear space sono le stesse e alcune sono contratte.
 For the case of $\trop \Grp(1,n)$ instead of  $T$ one  considers  the corresponding  tree with $n'\le n$ labelled leaves. In fact in this case the polyhedral complex dual to the subdivision has the coarsest polyhedral structure.

 In what follows we call an  \textit{open polyhedron}   a set of the form $P\setminus \partial P$ where  $P$ is a polyhedron and $\partial P$ is its boundary. For example the open  square with vertices $(0,0,1),(1,0,1),(1,0,0),$ and $(0,0,0)$ in $\mathbb R^3$ is an open polyhedron.

\begin{proof}[Proof of Theorem \ref{Fanstructure}]
We prove that $\F_d(\trop X)\cap \mathcal O$ can be written as the union of finitely many polyhedra, denoted by $F_T$, and hence the common refinement of these polyhedra  is the polyhedral complex structure on $\F_d(\trop X)\cap \mathcal O$.

There are two  key points in the proof. The first is that   the complement of a polyhedron is the union of open polyhedra and second that the projection of  an open polyhedron is an open polyhedron. Secondly it is crucial to describe  the polyhedral structure of a tropical linear space from its Pl\"ucker coordinates. In the following we will start by showing this last point.

Let $T$ be a combinatorial type of tropical linear spaces associated to the relative interior  of  a cone $C_T\subseteq \trop \Grp(d,n)\cap \mathcal O$. Consider $p\in C_T$ then  the tropical linear space $\Gamma_p$ is a subcomplex of the dual complex to a subdivision $T'$ of $P_M$, where $M$ is the uniform matroid associated to $\trop \Grp(d,n)\cap \mathcal O$ and $C_{T'}$ is a face of $C_T$. This implies that
$\Gamma_p=\coprod_i C_i(p) $ and  each cell $C_i(p)$
 in $ \trop\mathbb P^n$ has the following form 
 $$\{x\in \trop\mathbb P^n:  A(i,T) x^t\le  \textbf {f}(p) \text{ and }    B(i,T) x^t= \textbf {g}(p)  \}$$  
where $A(i,T) $  and $B(i,T)$ are matrices  with entries in $ \mathbb R$ and $\textbf{f}(p),\textbf{g}(p)$ are vectors whose entries are linear forms  in the coordinates of $p$, that depend only on $T$ and not on $p$. Note that if $p\in C_T\setminus C_T^{\circ}$ then $p\in C_{T'}\subset C_T$ hence   some of the $C_i(p)$ might be  the same. These are dual to the cell of $T'$ that is subdivided in $T$.

We are now ready to  define  $F_T$. This is  the set 
$$
F_T=\{p\in   C_T  : \Gamma_p\subseteq \trop X  \}
$$
hence 
	$$\F_d(\trop X)\cap \mathcal O=\bigcup_{T}F_T
$$
%	 F^{T}_{\mathcal S}	
	where the union is over all combinatorial types $T$ associated to the relative interior of the maximal cones of $\trop \Grp(d,n)\cap \mathcal O$.

The tropical linear space $\Gamma_p$ is contained in $\trop X$ if and only if for every $i$  we have $C^{\circ}_i(p)\subseteq \trop X$, that is
 
 $$
 F_T=\bigcap_i \{p\in C_T: C^{\circ}_i(p)\subseteq \trop X\}
 $$
where  $\Gamma_p=\coprod_i C^{\circ}_i(p)$ and  $C^{\circ}_i(p)$ is the relative interior of a   cell $C_i(p)$ of $\Gamma_p$. 
Denote by $F_{C_i}$ the set $\{p\in C_T: C^{\circ}_i(p)\subseteq \trop X\}$.  
We show that this set is open and is the union of open polyhedra.  

Consider the set 
 $$
 \tilde{F}_{C_i}:=\{(p,x)\in C_T\times \trop \mathbb P^{n} : x\in C^{\circ}_i(p)\text{ and } x\notin \trop X  \}\subseteq (\trop \Grp(d,n)\cap \mathcal O)\times \trop  \mathbb P^{n}.
 $$
 
Firstly we observe that  $x\notin\trop X$ if and only if $x$ is in the complement of any  cell of $\trop X$ that is
\begin{equation}\label{intsigma}
x\notin \trop X \Leftrightarrow x\in \bigcap_{\sigma \text{ cell of } \trop X} \sigma^{c}.
\end{equation}
The complement of a polyhedron is a union of open polyhedra hence  
the term on the left of  (\ref{intsigma}) is  the union of finitely many open polyhedra.

Since $$\tilde{F}_{C_i}=\{(p,x)\in C_T\times \trop \mathbb P^n :x\in C_i^{\circ}(p)\}  \cap \{(p,x)\in C_T\times \trop \mathbb P^n: 
x\notin \trop X \}$$ we obtain that $\tilde{F}_{C_i}$ is the union of finitely many open polyhedra. Moreover this is also the case for $\pi(\tilde{F}_{C_i})$ where $\pi$ is the projection $$\pi:\trop\Grp(d,n)\cap \mathcal O\times \trop \mathbb P^n\to \trop\Grp(d,n)\cap \mathcal O.$$ We can be describe $\pi(\tilde{F}_{C_i})$
in the following way 
$$ 
\pi(\tilde{F}_{C_i})=\{p\in C_T : \exists x\in C_i(p) \text{ such that } x\notin \trop X \}.
$$
The set $F_{C_i}$ is  the  complement of $\pi(\tilde{F}_{C_i})$ hence it is closed and it is  the union of finitely many polyhedra. This proves that $F_T$ is the union of finitely many polyhedra and hence the same holds for $\F_d(\trop X)\cap \mathcal O$.
\end{proof}

\begin{remark}\label{nonsurval}
It is not necessary to have a surjective  valuation $\val$.  Let $G=\val(\Bbbk)$ be the value group of $\val$ and assume $G\subsetneq \mathbb R$. Then for any   variety $X\subseteq \mathbb P^n$ we  have that each face of $\trop X$ is a $\val(\Bbbk)$-polyhedron, so it is defined by linear equalities and inequalities with coefficients in  $\val(\Bbbk)$. In particular if $\Gamma$ is a tropical linear space then the inequalities defining the cells have coefficients in  $\val(\Bbbk)$. This implies that the set $F_T$ is not a union of polyhedra but it is the intersection of this union with $\val(\Bbbk^*)^m$.  Let $\mathcal O$ be   an orbit of $\trop \mathbb P^{\binom{n+1}{d+1}-1}$ then  we  can define $\F_d(\trop X)\cap \mathcal O$ to be the Euclidean closure of $\bigcup_{T} F_T$.
\end{remark}

The structure of the tropical Fano scheme is strictly connected to the structure of the tropical variety $\trop X$.

\begin{proof}[Proof of Corollary \ref{fanProperty}]
The polyhedral structure on $ F_d(\trop X)\cap \mathcal O$ is the common refinement of the  $F_T$. In the case in which $\trop X\cap \mathcal O'$ is a fan we get that $F_T$ is the union of  finitely many cones for every $T$. This can be seen from the construction of each $F_T$ in the proof of Theorem \ref{Fanstructure}. Then the common refinement of these cones for every $T$ gives a fan structure on $ F_d(\trop X)\cap \mathcal O$ .
\end{proof}


\end{document}